\documentclass[12pt, a4paper]{amsart}
\usepackage{amsmath}
\usepackage{geometry,amsthm,graphics,tabularx,amssymb,shapepar}
\usepackage{amscd}
\usepackage[all,2cell,dvips]{xy}

\newcommand{\CF}{{\mathcal {F}}}

\newcommand{\CN}{{\mathcal {N}}}
\newcommand{\CO}{{\mathcal {O}}}

\newcommand{\CU}{{\mathcal {U}}}
\newcommand{\CV}{{\mathcal {V}}}

\newcommand{\CZ}{{\mathcal {Z}}}

\newcommand{\RD}{{\mathrm {D}}}

\newcommand{\End}{{\mathrm{End}}}

\newcommand{\GL}{{\mathrm{GL}}}

\newcommand{\rank}{{\mathrm{rank}}}

\newcommand{\SO}{{\mathrm{SO}}}

\newcommand{\Sp}{{\mathrm{Sp}}}

\newcommand{\tr}{{\mathrm{tr}}}

\newcommand{\vsp}{{\vspace{0.2in}}}

\newcommand{\con}{\textit{C}}

\newcommand{\ad}{\operatorname{ad}}

\newcommand{\diag}{\operatorname{diag}}

\newcommand{\oJ}{\operatorname{J}}

\newcommand{\os}{\operatorname{s}}

\newcommand{\oH}{\operatorname{H}}
\newcommand{\oO}{\operatorname{O}}

\newcommand{\oT}{\operatorname{T}}
\newcommand{\oU}{\operatorname{U}}
\newcommand{\oZ}{\operatorname{Z}}
\newcommand{\oD}{\textit{D}}

\newcommand{\sdim}{\operatorname{sdim}}
\newcommand{\rs}{\operatorname{s}}

\renewcommand{\a}{\mathfrak a}

\renewcommand{\u}{\mathfrak u}

\newcommand{\s}{\mathfrak s}

\renewcommand{\o}{\mathfrak o}

\newcommand{\z}{\mathfrak z}
\newcommand{\su}{\mathfrak s \mathfrak u}
\renewcommand{\sl}{\mathfrak s \mathfrak l}
\newcommand{\gl}{\mathfrak g \mathfrak l}

\newcommand{\Z}{\mathbb{Z}}
\newcommand{\C}{\mathbb{C}}
\newcommand{\R}{\mathbb R}

\newcommand{\K}{\mathbb{K}}

\newcommand{\la}{\langle}
\newcommand{\ra}{\rangle}

\newcommand{\be}{\begin {equation}}
\newcommand{\ee}{\end {equation}}
\newcommand{\bee}{\begin {equation*}}
\newcommand{\eee}{\end {equation*}}

\theoremstyle{Theorem}
\newtheorem{introtheorem}{Theorem}

\newtheorem{thm}{Theorem}[section]
\newtheorem{cort}[thm]{Corollary}
\newtheorem{lemt}[thm]{Lemma}

\newtheorem{thmt}[thm]{Theorem}

\theoremstyle{Theorem}
\newtheorem{lem}{Lemma}[section]

\newtheorem{leml}[lem]{Lemma}
\newtheorem{prpl}[lem]{Proposition}

\theoremstyle{Theorem}
\newtheorem{prp}{Proposition}[section]

\newtheorem{lemp}[prp]{Lemma}

\newtheorem{prpp}[prp]{Proposition}

\theoremstyle{Plain}

\theoremstyle{Definition}

\newtheorem{dfnp}[prp]{Definition}

\begin{document}

\title[Multiplicity one theorems]{Multiplicity one theorems: the Archimedean case}

\author{Binyong Sun}
\address{Academy of Mathematics and Systems Science\\
Chinese Academy of Sciences\\
Beijing, 100190, P.R. China} \email{sun@math.ac.cn}

\author{Chen-Bo Zhu}
\address{Department of Mathematics\\
National University of Singapore\\
Block S17, 10 Lower Kent Ridge Road\\
Singapore 119076} \email{matzhucb@nus.edu.sg}

\subjclass[2000]{22E30, 22E46 (Primary)}
\keywords{Classical groups, Jacobi groups, Casselman-Wallach representations, Gelfand-Kazhdan criteria, generalized functions}

\thanks{B. Sun was supported by NUS-MOE grant
R-146-000-102-112, and NSFC grants 10801126 and
10931006. C.-B. Zhu was supported by NUS-MOE grant
R-146-000-102-112.}


\begin{abstract}
Let $G$ be one of the classical Lie groups $\GL_{n+1}(\R)$, $\GL_{n+1}(\C)$, $\oU(p,q+1)$,
$\oO(p,q+1)$, $\oO_{n+1}(\C)$, $\SO(p,q+1)$, $\SO_{n+1}(\C)$,  and let $G'$
be respectively the subgroup $\GL_{n}(\R)$, $\GL_{n}(\C)$, $\oU(p,q)$, $\oO(p,q)$, $\oO_n(\C)$, $\SO(p,q)$, $\SO_n(\C)$, embedded in $G$ in the standard way. We show that every irreducible Casselman-Wallach representation of $G'$ occurs with multiplicity at most one in every irreducible  Casselman-Wallach representation of $G$. Similar results are proved for the Jacobi groups $\GL_{n}(\R)\ltimes \oH_{2n+1}(\R)$,
$\GL_{n}(\C)\ltimes \oH_{2n+1}(\C)$, $\oU(p,q)\ltimes \oH_{2p+2q+1}(\R)$, $\Sp_{2n}(\R)\ltimes \oH_{2n+1}(\R)$, $\Sp_{2n}(\C)\ltimes \oH_{2n+1}(\C)$, with their respective subgroups $\GL_{n}(\R)$,
$\GL_{n}(\C)$, $\oU(p,q)$, $\Sp_{2n}(\R)$, $\Sp_{2n}(\C)$.
\end{abstract}

\maketitle

\tableofcontents

\section{Introduction and main results}
\label{sec:intro}

Let $G$ be one of the (five plus two) classical groups
\begin{equation}
\label{classical}
  \GL_{n+1}(\R),\,\GL_{n+1}(\C),\,\oU(p,q+1),\,\oO(p,q+1),\, \oO_{n+1}(\C), \,\SO(p,q+1),\,\SO_{n+1}(\C),
\end{equation}
or one of the five Jacobi groups
\begin{equation}
\label{Jacobi}
  \GL_{n}(\R)\ltimes \oH_{2n+1}(\R),\, \GL_{n}(\C)\ltimes \oH_{2n+1}(\C),\,\oU(p,q)\ltimes \oH_{2p+2q+1}(\R),
\end{equation}
\[
  \Sp_{2n}(\R)\ltimes \oH_{2n+1}(\R),\,\Sp_{2n}(\C)\ltimes \oH_{2n+1}(\C), \quad p,q,n\geq 0.
\]
Here ``$\oH_{2k+1}$" indicates the appropriate Heisenberg group of dimension $2k+1$. A precise description of Jacobi groups is given in Section \ref{sec:proofC}.

Let $G'$ be respectively the subgroup
\[
  \GL_{n}(\R),\,\GL_{n}(\C),\,\oU(p,q),\,\oO(p,q),\, \oO_n(\C), \,\SO(p,q),\,\SO_n(\C),
\]
or
\[
  \GL_{n}(\R),\,\GL_{n}(\C),\,\oU(p,q),\,\Sp_{2n}(\R),\, \Sp_{2n}(\C),
\]
embedded in $G$ in the standard way. The main technical result of this paper is the following

\begin{introtheorem}
\label{thm:main} There exists a real algebraic anti-automorphism $\sigma$ on $G$ preserving $G'$ with the following property: every generalized function on $G$ which is invariant under the adjoint action of $G'$ is automatically $\sigma$-invariant.
\end{introtheorem}

A set of anti-automorphisms which satisfy Theorem \ref{thm:main} is
constructed in Section \ref{sec:proofC}. For example, the matrix
transpose is one such anti-automorphism when $G$ is a general linear
group.

By a representation of a Lie group, we mean a continuous linear
action of the group on a (complete, Hausdorff, complex) locally
convex topological vector space. When the Lie group is real
reductive, a representation is said to be a Casselman-Wallach
representation if it is Fr\'{e}chet, smooth, of moderate growth,
admissible and $\oZ$-finite. Here $\oZ$ is the center of the
universal enveloping algebra of its complexified Lie algebra. The
reader may consult \cite{Cass}, \cite[Chapter 11]{W2} or \cite{BK}
for more details about Casselman-Wallach representations.

By (a version of) the Gelfand-Kazhdan criterion (\cite[Corollary
2.5]{SZ08}), Theorem \ref{thm:main} for the classical groups implies
the following result (which we call the multiplicity one theorem for
Bessel models, as it implies uniqueness of the Bessel models
(\cite{GGP, JSZ1})).

\begin{introtheorem}
\label{thm:mainB}
Let $G$ be one of the classical groups in \eqref{classical}. Let $V$ (resp. $V'$) be an irreducible Casselman-Wallach representation of $G$ (resp. $G'$). Then the space of $G'$-invariant continuous bilinear functionals on $V\times V'$ is at
most one dimensional.
\end{introtheorem}

Theorem \ref{thm:mainB} and its p-adic analog have been expected
(by Bernstein, and Rallis) since the 1980's. When $V'$ is the trivial
representation, Theorem \ref{thm:mainB} is proved in \cite{AGS0},
\cite{AGS2} and \cite{Dijk}, in the case of general linear,
orthogonal, and unitary groups, respectively. The p-adic analog of
Theorem \ref{thm:mainB} is proved in \cite{AGRS} (except for the case of special orthogonal groups, which is proved in \cite{Wald09}). When the initial manuscript of this paper was completed, the authors learned that A. Aizenbud and D. Gourevitch had proved the multiplicity one theorems for the pairs $(\GL_{n+1}(\R),\GL_n(\R))$ and $(\GL_{n+1}(\C),\GL_n(\C))$, independently and in a different approach. This has since appeared as \cite{AG2}.

\vsp

Now assume that $G$ is one of the five Jacobi groups. Write
$G=G'\ltimes H$, where $H$ is an appropriate Heisenberg group. Fix a
non-trivial unitary character $\psi$ on the center of $H$. Let
$\widetilde G'$ be a double cover of $G'$ so that $\widetilde
G'\ltimes H$ admits a smooth oscillator representation $\omega_\psi$
corresponding to $\psi$, that is, $\omega_\psi$ is a genuine smooth
Fr\'{e}chet moderate growth representation of $\widetilde G'\ltimes
H$ which is irreducible with central character $\psi$ when
 viewed as a representation of $H$. We say that a representation of $G$
 is a Casselman-Wallach $\psi$-representation if it is of the form  $V=V_0\widehat \otimes \omega_\psi$
 (completed projective tensor product), where $V_0$ is a genuine Casselman-Wallach representation of $\widetilde G'$. This definition is independent of $\widetilde G'$ and $\omega_\psi$. The representation $V$ is irreducible if and only if $V_0$ is. See \cite{Su10} for more details on Casselman-Wallach $\psi$-representations.

By the Gelfand-Kazhdan criterion for Jacobi groups (\cite[Corollary
D]{Su10}), Theorem \ref{thm:main} for the Jacobi groups implies the
following result (which we call the multiplicity one theorem for
Fourier-Jacobi models, as it implies uniqueness of the
Fourier-Jacobi models (cf. \cite{GGP})).

\begin{introtheorem}
\label{thm:mainC}
Let $G$ be one of the Jacobi groups in \eqref{Jacobi}. Let $V$ be an irreducible Casselman-Wallach $\psi$-representation of $G$, and let $V'$ be an irreducible Casselman-Wallach representation of $G'$. Then the space of $G'$-invariant continuous bilinear functionals on $V\times V'$ is at
most one dimensional.
\end{introtheorem}

The p-adic analog of Theorem \ref{thm:mainC} was conjectured by D. Prasad (\cite[Page 20]{Pr96} in the case of symplectic groups), and is proved in \cite{Su09}.

\section{A uniform formulation}
\label{sec:formulation}

We first introduce some general notation which will be used throughout the paper.
For any (smooth) manifold $M$, denote by $\con^{-\infty}(M)$ the space of generalized functions on $M$, which by definition consists of continuous linear functionals on $\RD_c^\infty(M)$,
the space of (complex) smooth densities on $M$ with compact
supports. The latter is equipped with the usual inductive smooth
topology. For any locally closed subset $Z$ of $M$, denote by
\begin{equation}
  \label{dcmz} \con^{-\infty}(M;Z) \subset \con^{-\infty}(U)
\end{equation}
the subspace consisting of all $f$ which are supported in $Z$, where
$U$ is an open subset of $M$ containing $Z$ as a closed subset.
This definition is independent of $U$.

If $M$ is a Nash manifold, denote by
$\con^{-\xi}(M)\subset\con^{-\infty}(M)$ the space of tempered
generalized functions on $M$, and by $\con^{\,\varsigma}(M)\subset
\con^{-\xi}(M)$ the space of Schwartz functions. We refer the
interested reader to \cite{Sh, AG1} on generalities of Nash
manifolds and their function spaces. Since the closure of every
semialgebraic set is semialgebraic, given any locally closed
semialgebraic subset $Z$ of a Nash manifold $M$, we may find an open
semialgebraic subset $U$ of $M$ containing $Z$ as a closed subset.
We define $\con^{-\xi}(M;Z)$ as the subspace of $\con^{-\xi}(U)$
consisting of all $f$ which are supported in $Z$. Again this is
independent of $U$.

If $H$ is a Lie group
acting smoothly on a manifold $M$, then for any character $\chi_H$
of $H$, denote by \begin{equation} \label{dchi}
\con^{-\infty}_{\chi_H}(M)\subset \con^{-\infty}(M) \end{equation}
the subspace consisting of all $f$ which are $\chi_H$-equivariant,
i.e.,
\[
   f(h\cdot x)=\chi_H(h)f(x),\quad \textrm{for all } h\in
   H.
\]
Similar notation (such as $\con^{-\xi}_{\chi_H}(M;Z)$) will be used without further explanation.

\vsp

We now proceed to describe a general set-up in order to work with all classical groups in a uniform manner.

Let $A$ be a finite dimensional semi-simple commutative algebra
over $\R$, which is thus a finite product of copies of $\R$ and
$\C$. Let $\tau$ be a $\R$-algebra involution on $A$. We call $(A,\tau)$ (or $A$ when $\tau$ is understood) a commutative involutive algebra (over $\R$).  Let $\epsilon=\pm 1$. Let $E$ be an $\epsilon$-Hermitian $A$-module, namely, it is a finitely generated $A$-module, equipped with a non-degenerate $\R$-bilinear map
\[
  \la\,,\,\ra_E:E\times E\rightarrow A
\]
satisfying
\[
     \la u,v\ra_E=\epsilon \la v,u\ra_E^\tau, \quad \la au,v\ra_E=a\la u,
     v\ra_E,\quad a\in A,\, u,v\in E.
\]
Denote by $\oU(E)$ the group of all $A$-module automorphisms of $E$
which preserve the form $\la\,,\,\ra_E$, and by $\u(E)$ its Lie
algebra, which consists of all $x\in \End_A(E)$ such that
\[
  \la xu, v\ra_E+\la u, xv\ra_E=0,\quad u,v\in E.
\]

Write $E_\R:=E$, viewed as a real vector space. Following
Moeglin-Vigneras-Waldspurger (\cite{MVW87}), we define a subgroup
\begin{equation}\label{utilde}
       \breve{\oU}(E)\subset\GL(E_\R)\times \{\pm 1\}
\end{equation}
consisting of pairs $(g,\delta)$ such that either
\[
  \delta=1 \quad\textrm{and}\quad \la gu,gv\ra_E=\la u,v\ra_E,\,\,\,  u,v\in E,
\]
or
\[
     \delta=-1 \quad\textrm{and}\quad
     \la gu,gv\ra_E=\la v,u\ra_E,\,\,\,  u,v\in E.
  \]
Note that for every element
$(g,\delta)\in \breve{\oU}(E)$, if $\delta=1$, then $g$ is
automatically $A$-linear, and if $\delta=-1$, then $g$ is
$\tau$-conjugate linear. Denote by
\begin{equation}
\label{dctilde}
  \chi_E: \breve{\oU}(E)\rightarrow \{\pm 1\}
\end{equation}
the quadratic character of $\breve{\oU}(E)$ projecting to the second
factor. It is a surjective homomorphism with kernel $\oU(E)$.

Let $\breve{\oU}(E)$ act on $\oU(E)$ by
\begin{equation}\label{actiong}
  (g,\delta)\cdot x:=gx^\delta g^{-1},
\end{equation}
and act on $\u(E)$ through its differential, i.e.,
\[
  (g,\delta)\cdot x:=\delta gxg^{-1}.
\]
It is known that every $\oU(E)$-orbit in $\oU(E)$ or $\u(E)$ is  $\breve{\oU}(E)$-stable (\cite[Proposition 4.I.2]{MVW87}). Let $\breve{\oU}(E)$ act on $E$ by
\begin{equation}\label{actione}
  (g,\delta)\cdot v:=\delta gv,
\end{equation}
and act on  $\oU(E)\times E$ and $\u(E)\times E$ diagonally.

The next five sections will be devoted to a proof of the following

\begin{thm}\label{liealg}
One has that
\[
   \con^{-\xi}_{\chi_E}(\u(E)\times E)=0.
\]
\end{thm}

\section{Fourier transform and rigidity}
\label{sec:rigid}

Let $F$ be a finite dimensional real vector space, which is
canonically a Nash manifold. Denote by $\C[F]$ the algebra of
(complex) polynomial functions and $\oD[F]$ the algebra of constant
coefficient differential operators, on $F$. It is a classical result
of L. Schwartz that
\[
  \con^{-\xi}(F; \{0\})=\oD[F] \delta_F.
\]
Here $\delta_F$ is a Dirac function on $F$, which is characterized
(up to a nonzero scalar) by the equation
\[
  \lambda \delta_{F}=0, \quad \textrm{for all real linear functionals $\lambda$ on $F$.}
  \]
This is the simplest case of rigidity we have in mind.

\vsp From now on, we assume that $F$ is equipped with a
non-degenerate  bilinear form $\la\,,\,\ra_F$ which is either symmetric or skew-symmetric.

We introduce one general notation which will be extensively used. If $M$ is a Nash manifold, we define the partial Fourier transform along $F$ to be the topological linear
automorphism
\[
  \CF_F:\con^{\,\varsigma}(M\times F)\rightarrow
    \con^{\,\varsigma}(M\times F)
\]
given by
\begin{equation}
\label{dpFourier}
  (\CF_F f)(m, y)=\int_F f(m,x)e^{2\pi \sqrt{-1}\,\la x,y\ra_F}\,dx, \ \ m\in M, \ y\in F.
\end{equation}
Here $dx$ is any fixed Lebesgue measure on $F$. The partial Fourier
transform uniquely extends to a topological linear isomorphism
\[
  \CF_F: \con^{-\xi}(M\times F)\rightarrow
\con^{-\xi}(M\times F).
\]
When $M$ reduces to a single point, we recover the usual Fourier transform.

When the factor $F$ is understood, for any two closed semialgebraic subsets $Z_1$ and $Z_2$ of $M\times F$, put
\begin{equation}
\label{double-support}
  \con^{-\xi}(M\times F;Z_1,Z_2):=\{f\in \con^{-\xi}(M\times F; Z_1)\mid \CF_F (f)\in \con^{-\xi}(M\times F; Z_2)\}.
\end{equation}

For a subspace $F'$ of $F$, let $F'^\perp$ denote its
perpendicular space:
\[
 F'^\perp:=\{v\in F\mid \la v', v\ra_F=0, \,v'\in
 F'\}.
\]

\begin{lem}\label{lemf1}
If $F=F'\oplus F''$ is a direct sum decomposition, then
\[
   \con^{-\xi}(F;F',F'^\perp)=\C[F']\otimes \con^{-\xi}(F'';\{0\}).
\]
\end{lem}

\begin{proof} Note that every tempered generalized function has a
finite order. Hence by the well-known result of L. Schwartz about
local representation of a generalized function with support, we
have
\[
  \con^{-\xi}(F;F')=\con^{-\xi}(F')\otimes \con^{-\xi}(F'';\{0\}).
\]
The lemma then follows easily.
\end{proof}

For later use, we record the following
\begin{prpl}\label{lemf3}
If $F^0$ is a non-degenerate subspace of $F$, and
\[
  (F^0)^\perp=F^+\oplus F^-
\]
is a decomposition into totally isotropic subspaces $F^+$ and $F^-$, then
\[
 \con^{-\xi}(F;F^+\oplus F^0,F^+\oplus F^0)=\C[F^+]\otimes \con^{-\xi}(F^-;\{0\})
 \otimes \con^{-\xi}(F^0).
\]
\end{prpl}
\begin{proof}
The proof is similar to that of Lemma \ref{lemf1}.
\end{proof}

We also need the following result, which is a special case of \cite[Theorem A]{SZ10}.

\begin{prpl}\label{quad}
Assume $\dim_\R F=2k$. Let $F_1,F_2,\cdots, F_s$ be a set of
(distinct) totally isotropic subspaces of $F$, each of dimension
$k$. Then
\[
   \con^{-\xi}(F;F_1\cup F_2\cup \cdots \cup  F_s, F_1\cup F_2\cup \cdots \cup  F_s)
   =\bigoplus_{i=1}^s \con^{-\xi}(F;F_i,
   F_i).
\]
\end{prpl}

\section{Nonnegativity of eigenvalues of an Euler vector field}
\label{sec:eigenvalue}

We continue with the notation of Section \ref{sec:formulation}. Set
\[
  \oU(A):=\{a\in A^\times\mid a^\tau a=1\},
\]
and its Lie algebra
\[
  \u(A):=\{a\in A\mid a^\tau+ a=0\}.
\]
Scalar multiplication then yields a homomorphism $\oU(A)\rightarrow
\oU(E)$ and its differential $\u(A)\rightarrow \u(E)$. Denote by
$\oZ(E)$ and $\z(E)$ their respective images. Then $\oZ(E)$
coincides with the center of $\oU(E)$ (but $\z(E)$ may not coincide
with the center of $\u(E)$).

Denote by
\[
  \tr_A: \End_A(E)\rightarrow A
\]
the trace map. It is specified by requiring that the diagram
\[
  \begin{CD}
            \End_A(E)@>\tr_A>> A\\
            @V 1_{A_0}\otimes VV           @VVV\\
            \End_{A_0}(A_0\otimes_A E)@>\tr>> A_0
  \end{CD}
\]
commutes for every quotient field $A_0$ of $A$, where the bottom
arrow is the usual trace map. Set
\[
  \sl(E):=\{x\in \End_A(E)\mid \tr_A(x)=0\}.
\]
Then we have
\begin{equation}\label{decomgl}
  \End_A(E)=\{\textrm{scalar multiplication by $a\in A$}\}\oplus \sl(E),
\end{equation}
and
\[
   \u(E)=\z(E)\oplus \su(E),
\]
where $\su(E):=\u(E)\cap \sl(E)$.

\vsp
We call the commutative involutive algebra $A$ simple if it is nonzero, and
has no $\tau$-stable ideal except $\{0\}$ and itself. Every
simple commutative involutive algebra is isomorphic to one of the
followings:
\begin{equation}\label{five}
   (\R, 1),\, \,(\C,1),\,\, (\C, \overline{\phantom{a}}\,),\,\, (\R\times \R,\tau_\R),
   \,\,(\C\times
   \C,\tau_\C),
\end{equation}
where $\tau_\R$ and $\tau_\C$ are the maps which interchange the
coordinates.

Assume in the rest of the section that $A$ is simple. We say that
$(A,\tau; \epsilon)$ is of orthogonal type if it is either $(\R, 1;
1)$ or $(\C,1;1)$. If $(A,\tau; \epsilon)$ is not of orthogonal
type, we fix a nonzero element $c_0\in A$ so that
\[
  c_0+\epsilon c_0^\tau=0.
\]
For any $v\in E$, write
\[
  \phi_v(u):=\la u,v\ra_E\, v, \quad u\in E,
\]
then $\phi_v\in \End_A(E)$.  Denote by $\phi'_v\in \sl(E)$ the
projection of $\phi_v$ to the second factor according to the
decomposition \eqref{decomgl}. For any $x\in \su(E)$, set
\begin{equation}\label{psixv}
  \phi_{x,v}:=\left\{
                     \begin{array}{ll}
                         x\phi_v+\phi_v x,\quad&\textrm{if  $(A,\tau; \epsilon)$ is of orthogonal type,}\\
                         c_0\,\phi'_v, \quad&\textrm{otherwise.}\smallskip\\

                     \end{array}
              \right.
\end{equation}
This is checked to be in $\su(E)$.

Recall that an element of $\u(E)$ is said to be
nilpotent (semisimple) if it is nilpotent (semisimple) as a $\R$-linear operator on $E$. Recall also that a nilpotent element of $\u(E)$ (which is automatically in $\su(E)$) is said to be distinguished if it commutes with no nonzero semisimple element
in $\su(E)$ (cf. \cite[Section 8.2]{CM}).

Fix a distinguished nilpotent element $\mathbf e\in \su(E)$.  Following \cite{AGRS}, we
define
\begin{equation}\label{dee}
  E(\mathbf e):=\{\,v\in E\mid \phi_{\mathbf e,v}\in [\su(E),\mathbf e]\,\}.
\end{equation}

Extend $\mathbf e$ (by Jacobson-Morozov Theorem) to a standard triple $\mathbf h, \mathbf e, \mathbf f$ in $\su(E)$ so that
\[
  [\mathbf h, \mathbf e]=2\mathbf e, \quad
  [\mathbf h,\mathbf f]=-2\mathbf f,\quad [\mathbf e, \mathbf f]=\mathbf h.
\]
Denote by $E_{\mathbf{h}}^i\subset E$ the
eigenspace of $\mathbf{h}$ with eigenvalue $i$, where $i\in \Z$.
Write
\[
   E_{\mathbf{h}}^+:=\bigoplus_{i> 0} E_{\mathbf{h}}^i,\quad  \textrm{and}\quad
    E_{\mathbf{h}}^-:=\bigoplus_{i<0} E_{\mathbf{h}}^i.
\]
Note that $\la E_{\mathbf{h}}^i, E_{\mathbf{h}}^j\ra_E=0$ whenever $i+j\neq 0$.

\begin{leml}
\label{Ee}
If $A$ is a field, then $E(\mathbf e)=E_{\mathbf
h}^{+}\oplus E_{\mathbf h}^0$.
\end{leml}
\begin{proof}
We prove the lemma in the case of real orthogonal groups. The other
cases are proved similarly. So assume that
$(A,\tau;\epsilon)=(\R,1;1)$. Then $\su(E)=\o(E)$ is a real
orthogonal Lie algebra.

View $E$ as a $\sl_2(\R)$-module via the standard triple. Let
\begin{equation}\label{dece}
   E=E_1\oplus E_2\oplus\cdots \oplus E_s
\end{equation}
be a decomposition of $E$ into irreducible $\sl_2(\R)$-modules. By the classification of distinguished nilpotent orbits (\cite[Theorem 8.2.14]{CM}), we know that
(\ref{dece}) is an orthogonal decomposition, and $E_1, E_2,
\cdots, E_s$ have pairwise distinct odd dimensions. Denote by
$\mathbf{e}_i\in \o(E_i)\subset \o(E)$ the restriction of $\mathbf e$ to
$E_i$.

View $\o(E)$ as a real  quadratic space under the trace form.  For every $v\in E$, we have that $v\in E(\mathbf e)$ if and only if
\begin{eqnarray*}
&&\phi_{\mathbf e,v}\in [\o(E), \mathbf e] \Leftrightarrow \phi_{\mathbf e,v}\perp [\o(E), \mathbf e]^\perp\\
&&\phantom { \phi_{\mathbf e,v}\in [\o(E), \mathbf e] } \Leftrightarrow  \phi_{\mathbf e,v}\perp \o(E)^{\mathbf e} \,\,\textrm{ (the centralizer of $\mathbf e$ in $\o(E)$)}\\
&&\phantom { \phi_{\mathbf e,v}\in [\o(E), \mathbf e] } \Leftrightarrow  \la \mathbf e v, x v\ra_E=0\,\, \textrm{ for all } x\in \o(E)^{\mathbf e}.
\end{eqnarray*}
Thus if $v\in E(\mathbf e)$, then we have $\la \mathbf e v, \mathbf{e}_i^{2k+1}  v\ra_E=0$ for all $1\leq i\leq s$ and $k\geq 0$, and so $v$ must be in $E_{\mathbf h}^{+}+E_{\mathbf h}^0$.

On the other hand, every element $x\in \o(E)^{\mathbf e}$ stabilizes $E_{\mathbf h}^{+}+E_{\mathbf h}^0$. Therefore $v\in E_{\mathbf h}^{+}+E_{\mathbf h}^0$ implies that $\la \mathbf e v, x v\ra_E=0$. This finishes the proof.
\end{proof}

Denote by
\begin{equation}
\label{dnull}
  \Gamma_E:=\{v\in E\mid \la v,v\ra_E=0\}
\end{equation}
the null cone of $E$. Equip $E_\R=E$ with the (symmetric or
skew-symmetric) bilinear form
\[
  \la u,v\ra_{E_\R}:=\tr_{A/\R}(\la u,v\ra_E), \quad u,v\in E,
\]
where $\tr_{A/\R}: A\rightarrow \R$ is the usual trace map.

Put
\begin{equation} \label{dve}
  \CV_{E,\mathbf e}:=\con^{-\xi}(E; E(\mathbf e)\cap\Gamma_E, E(\mathbf e)\cap
  \Gamma_E)^{\oZ(E)},
\end{equation}
where and as usual, a superscript by a group indicates the group
invariants. This space arises naturally when one carries out
the reduction within the null cone. See Lemma \ref{support}.

For any finite dimensional real vector space $F$ and any $x\in
\End_\R(F)$, denote by $\epsilon_{F,x}$ the vector field on $F$
whose tangent vector at $v\in F$ is $xv$. When $x=1$ is the identity
operator, this is the usual Euler vector field
$\epsilon_F:=\epsilon_{F,1}$.

The main result of this section is the following

\begin{prpl}\label{glk}
The vector field $\epsilon_{E,\mathbf h}$ acts semisimply on
$\CV_{E,\mathbf e}$, and all its eigenvalues are nonnegative
integers.
\end{prpl}

If $A$ is a field, then
\begin{eqnarray*}
  &&\CV_{E,\mathbf e}\subset\con^{-\xi}(E; E(\mathbf e), E(\mathbf e))\\
  &&\phantom{\CV_{E,\mathbf e}}= \con^{-\xi}(E; E_{\mathbf h}^{+}\oplus E_{\mathbf h}^0,E_{\mathbf h}^{+}\oplus
  E_{\mathbf h}^0)\quad (\textrm{Lemma \ref{Ee}})\\
  &&\phantom{\CV_{E,\mathbf e}}=\C[E_{\mathbf h}^{+}]\otimes \con^{-\xi}(E_{\mathbf h}^{-};\{0\})\otimes \con^{-\xi}(E_{\mathbf h}^0) \quad (\textrm{Proposition \ref{lemf3}}),
\end{eqnarray*}
and Proposition \ref{glk} follows easily.

Otherwise assume that $(A,\tau)=(\K\times \K, \tau_\K)$, where
$\K=\R$ or $\C$. Up to an isomorphism, every $\epsilon$-Hermitian
$A$-module is of the form
\[
  (E,\la\,,\,\ra_E)=(\K^n\oplus \K^n, \la\,,\,\ra_{\K,n}), \quad n\geq 0,
\]
where $\K^n$ is considered as
a space of column vectors, and the $\epsilon$-Hermitian form $\la\,,\,\ra_{\K,n}$ is given
by
\[
  \left(\left[
          \begin{array}{c}
           u\\
           v
          \end{array}
  \right],
   \left[\begin{array}{c}
           u'\\
           v'
         \end{array}
  \right]\right)\mapsto ({v'}^t u, \epsilon {u'}^t v).
\]
Then
\[
  \oU(E)=\left\{\left[\begin{array}{cc}
                   g&0\\
                   0&g^{-t}\\
              \end{array}\right]\mid g\in
              \GL_n(\K)\right\}=\GL_n(\K),
\]
and
\[
  \u(E)=\left\{\left[\begin{array}{cc}
                   x&0\\
                   0&-x^t\\
              \end{array}\right]\mid x\in
              \gl_n(\K)\right\}=\gl_n(\K).
\]

In this case every distinguished nilpotent element of $\su(E)$ is
principal and so we may assume that
\[\mathbf e=\left[\begin{smallmatrix}0&1&0&\cdots &0&0\\
                                0&0&1&\cdots &0&0\\
                                 &&\cdots&\cdots&&\\
                                0&0&0&\cdots &0&1\\
                                0&0&0&\cdots &0&0\end{smallmatrix}\right]
                                \]
and
\[
 \mathbf h=\diag(n-1, n-3, \cdots, 3-n,1-n).
\]
As in the proof of Lemma \ref{Ee}, one easily calculates $E(\mathbf e)$ and one has
\begin{equation}\label{deffi}
  E(\mathbf e)\cap\Gamma_E=\cup_{i=0}^n F_i,\,\,\textrm{
  where }\,
   F_i=(\K^i\oplus\{0\}^{n-i})\oplus (\{0\}^i\oplus \K^{n-i}).
\end{equation}
Proposition \ref{quad} implies that
\[
  \CV_{E,\mathbf e}=\con^{-\xi}(E; E(\mathbf e)\cap\Gamma_E, E(\mathbf e)\cap
  \Gamma_E)^{\oZ(E)}=\bigoplus_{i=0}^n\con^{-\xi}(E; F_i,F_i)^{\oZ(E)}.
\]

To finish the proof of Proposition \ref{glk}, it therefore suffices to prove the following

\begin{leml}\label{egl}
The vector field $\epsilon_{E,\mathbf h}$
acts semisimply on $\con^{-\xi}(E; F_i,F_i)^{\oZ(E)}$, and all its eigenvalues are nonnegative integers ($0\leq i\leq n$).
\end{leml}

\begin{proof}
We prove the lemma for $\K=\R$. The complex case is proved in the
same way.

Denote by $x_1,x_2,\cdots,x_n, \,y_1,y_2,\cdots, y_n$ the standard
coordinates of $\R^n\oplus \R^n$, and write $
  \partial_j=\frac{\partial}{\partial x_j}$ and $d_j=\frac{\partial}{\partial
  y_j}$, for $j=1,2,\cdots, n$.

By Lemma \ref{lemf1}, the space $\con^{-\xi}(E;F_i,F_i)$ has a
basis consisting of generalized functions of the form
\begin{eqnarray*}
  &&f=x_1^{a_1}x_2^{a_2}\cdots x_i^{a_i}\, y_{i+1}^{b_{i+1}} y_{i+2}^{b_{i+2}} \cdots y_n^{b_n} \\
  &&\quad \otimes  \partial_{i+1}^{a_{i+1}-1} \partial_{i+2}^{a_{i+2}-1} \cdots
  \partial_n^{a_n-1}\, d_1^{b_1-1}d_2^{b_2-1}\cdots d_i^{b_i-1}
  \delta_{F'_i},
\end{eqnarray*}
where $a_1,...,a_i,b_{i+1},...,b_n$ are nonnegative integers, and
the rest of $a$'s and $b$'s are positive integers. Here $\delta_{F'_i}$ is a fixed Dirac function on the space
\[
  F'_i:=(\{0\}^i\oplus \R^{n-i})\oplus (\R^i\oplus\{0\}^{n-i}), \quad \textrm{(a complement of $F_i$).}
\]

The generalized function $f$ as above is an eigenvector for both
$\oZ(E)$ and $\epsilon_{E,\mathbf h}$. For $f$ to be
$\oZ(E)$-invariant, we must have
\[
   \sum_{j\leq i} (a_j+b_j)=\sum_{j> i} (a_j+b_j).
\]
Then the $\epsilon_{E,\mathbf h}$-eigenvalue of $f$ is
\begin{eqnarray*}
  &&\sum_{j\leq i}(n-(2j-1))a_j - \sum_{j>i}(n-(2j-1))a_j\\
  &&+\sum_{j\leq i}(n-(2j-1))b_j - \sum_{j>i}(n-(2j-1))b_j\\
    &&\!\geq \!(n-2i)(\sum_{j\leq i}a_j)- (n-2i)(\sum_{j>i}a_j)+(n-2i)(\sum_{j\leq i}b_j)- (n-2i)(\sum_{j>i}b_j)\\
    &&=0.
\end{eqnarray*}
\end{proof}

\section{Reduction within the null cone: distinguished orbits}
\label{sec:nullcone}

We continue to assume that $(A,\tau)$ is simple. Since Theorem \ref{liealg} is trivial when $E=0$, we will assume that $E$ is nonzero. Denote by $\CN_E\subset \su(E)$ the null cone, which consists of all nilpotent elements in $\su(E)$. Let
\begin{equation}\label{filtn}
  \CN_E=\CN_0\supset \CN_1\supset \cdots \supset \CN_r=\{0\}\supset
  \CN_{r+1}=\emptyset
\end{equation}
be a filtration of $\CN_E$ by its closed subsets so that each
difference
\[
  \CO_i:=\CN_i\setminus \CN_{i+1},\quad 0\leq i\leq r,
\]
is a $\breve{\oU}(E)$-orbit (which is also a $\oU(E)$-orbit). In this and the next section, we shall prove the following
\begin{prp}\label{indn} If every element of $\con^{-\xi}_{\chi_E}(\su(E)\times
E)$ is supported in $\CN_i\times \Gamma_E$, for some fixed $0\leq i\leq r$,
then every element of $\con^{-\xi}_{\chi_E}(\su(E)\times
E)$ is supported in $\CN_{i+1}\times \Gamma_E$.
\end{prp}
\vsp

For the ease of notation, denote $\s:=\su(E)$. We shall view $\s$ as a non-degenerate real quadratic space via the form
\[
  \la x,y\ra_\s:=\tr_{A/\R}(\tr_A(xy)).
\]

It is easily verified (and important for us) that the partial
Fourier transforms $\CF_E$ and $\CF_\s$ both preserve the space
$\con^{-\xi}_{\chi_E}(\s\times E)$.

\begin{lemp}
Proposition \ref{indn} holds when $\s=0$.
\end{lemp}
\begin{proof} For $\s=0$, the assumption of Proposition \ref{indn} implies that $\con^{-\xi}_{\chi_E}(\s \times E)\subset \con^{-\xi}(E;\Gamma_E,\Gamma_E)^{\oZ(E)}$.
The latter space is easily checked to be zero.
\end{proof}

\vsp For the remaining part of this section, assume that $\s\neq 0$.

Before proceeding further, we introduce a version of pull back of
generalized functions.

\begin{dfnp}\label{submersive}
Let $Z$ and $Z'$ be locally closed subsets of manifolds $M$ and
$M'$, respectively. A smooth map $\phi:M\rightarrow M'$ is said to
be submersive from $Z$ to $Z'$ if
\begin{itemize}
  \item
    $\phi$ is submersive at every point of $Z$, and
  \item
     for every $z\in Z$, there is an open neighborhood $U$ of $z$
in $M$ such that
\[
   \phi^{-1}(Z')\cap U=Z\cap U.
\]
\end{itemize}
\end{dfnp}

The following lemma is elementary.

\begin{lemp}\label{lpullback}
If  $\phi:M\rightarrow M'$ is submersive from $Z$ to $Z'$, as in
Definition \ref{submersive}, then there is a unique linear map
\begin{equation}\label{pullback}
  \phi^*: \con^{-\infty}(M';Z')\rightarrow
  \con^{-\infty}(M;Z),
\end{equation}
with the following property: for any open subset $U$ of $M$ and
$U'$ of $M'$, if
\begin{itemize}
 \item
 $\phi$ restricts to a submersive map $\phi_U: U\rightarrow U'$,
\item $Z'\cap U'$ is closed in $U'$, and \item
 $
  \phi_U^{-1}(Z'\cap U')=Z\cap U,
 $
\end{itemize}
then the diagram
\[
  \begin{CD}
             \con^{-\infty}(M';Z')@>\phi^*>>\con^{-\infty}(M;Z)\\
            @VVV           @VVV\\
            \con^{-\infty}(U';Z'\cap U') @> \phi_U^*>> \con^{-\infty}(U;Z\cap U)
  \end{CD}
\]
commutes, where the two vertical arrows are restrictions, and the
bottom arrow is the usual pull back map of generalized functions
via a submersion.
\end{lemp}

The map $\phi^*$ in (\ref{pullback}) is still called the pull back.
It is injective if $\phi(Z)=Z'$. In this case, we say that $\phi$ is
submersive from $Z$ onto $Z'$. If $M$, $M'$ are Nash manifolds, $Z$,
$Z'$ are locally closed semialgebraic subsets, and $\phi$ is a Nash
map which is submersive from $Z$ to $Z'$, then $\phi^*$ maps
$\con^{-\xi}(M';Z')$ into $\con^{-\xi}(M;Z)$ (cf. \cite[Section
B.2]{AGS1}).

\vsp

We continue the proof of Proposition \ref{indn}.

Fix $i\in \{0,1,\cdots,r\}$ and assume that $\CO_i$ is
distinguished, namely, some (so all) elements of it are
distinguished. We use the notation of last section. Put
\begin{equation}
   \CZ_i:=(\CN_{i+1}\times \Gamma_E)\cup (\bigsqcup_{\mathbf{e}\in \CO_i}
  \{\mathbf{e}\}\times  (E(\mathbf e)\cap\Gamma_E)).
\end{equation}
One checks that $\CZ_i$ is a closed  semialgebraic subset of
$\s\times E$.

\begin{lemp}\label{support} Assume that every element of
$\con^{-\xi}_{\chi_E}(\s\times E)$ is supported in $\CN_i\times
\Gamma_E$. Then every $f\in \con^{-\xi}_{\chi_E}(\s\times E)$ is
supported in $\CZ_i$.
\end{lemp}

\begin{proof} We follow the method of \cite{AGRS}.
For every $t\in \R$, define a map
\[
 \begin{array}{rcl}
  \eta:=\eta_t: \s\times E &\rightarrow &\s\times E,\\
                 (x,v)&\mapsto& (x-t\phi_{x,v},v),
 \end{array}
\]
which is checked to be submersive from $\s\times \Gamma_E$ to
$\s\times \Gamma_E$. Therefore, by Lemma \ref{lpullback}, it
yields a pull back map
\[
   \begin{array}{rcl}
   \eta^*: \con^{-\infty}(\s\times E;\s\times \Gamma_E) &\rightarrow
   &\con^{-\infty}(\s\times E;\s\times \Gamma_E).\\
  \end{array}
\]

Fix $f\in\con^{-\xi}_{\chi_E}(\s\times E)$. By our assumption,
\[
  f\in \con^{-\xi}_{\chi_E}(\s\times E;\CN_{i}\times \Gamma_E)
  \subset \con^{-\xi}_{\chi_E}(\s\times E;\s\times \Gamma_E).
\]
Since the map $\eta$ is algebraic and
$\breve{\oU}(E)$-equivariant,
\[
  \eta^* (f)\in \con^{-\xi}_{\chi_E}(\s\times E;\s\times \Gamma_E).
\]

Let $(\mathbf{e},v)\in
\CO_{i}\times \Gamma_E$ be a point in the support of $f$. It is routine to check that $\eta$ restricts to a bijection from
$\s\times \Gamma_E$ onto itself.   Denote
by
\[
  \mathbf e':=\mathbf e'(\mathbf e,v,t)\in \s
\]
the unique element so that
\[
  \eta(\mathbf{e}',v)=(\mathbf{e},v).
\]
Then $(\mathbf e',v)$ is in the support of $\eta^*(f)$, and
therefore our assumption implies that
\begin{equation}\label{eprime}
  \mathbf e'\in \CN_i.
\end{equation}

An easy calculation shows that
\[
  \mathbf{e}'=\left\{\begin{array}{ll}
             \mathbf{e}+t \phi_{\mathbf{e},v}+ t^2 \phi_v\mathbf{e} \phi_v,\quad&
      \textrm{if $(A,\tau;\epsilon)$ is of othogonal type},\smallskip\\        \mathbf{e}+t\phi_{\mathbf{e},v},\quad&\textrm{otherwise.}\\
                \end{array}
      \right.
\]
Since $\CO_i$ is open in $\CN_i$, (\ref{eprime}) implies that
\[
  \phi_{\mathbf{e},v}=\frac{d}{dt}|_{t=0}\, \mathbf{e}'(\mathbf e,v,t)\in
  \oT_{\mathbf{e}}(\CO_i)=[\u(E), \mathbf{e}]=[\s,\mathbf{e}],
\]
i.e., $v\in E(\mathbf e)$, and the proof is now complete.
\end{proof}

\vsp Denote by
\begin{equation}
\label{dcv}
  \CV_{\s\times E,\CO_i}\subset \con^{-\xi}(\s\times E; \CO_i\times
  E)^{\oU(E)}
\end{equation}
the subspace consisting of those $f$ such that both $f$ and
$\CF_E(f)$ are supported in $\bigsqcup_{\mathbf{e}\in \CO_i}
\{\mathbf{e}\}\times (E(\mathbf{e})\cap \Gamma_E)$.

\begin{prpp}\label{negative}
The Euler vector field $\epsilon_{\s}$ acts
semisimply on $\CV_{\s\times E,\CO_i}$, and all its eigenvalues are
real numbers $<-\frac{1}{2} \dim_\R \s$.
\end{prpp}

Let us first prove the following
\begin{lemp}
Proposition \ref{negative} implies Proposition \ref{indn} when
$\CO_i$ is distinguished.
\end{lemp}
\begin{proof}
Denote by $q_\s$ the quadratic form on $\s$, i.e.,
\[
  q_{\s}(x)=\la x,x\ra_\s= \tr_{A/\R}(\tr_A(x^2)).
\]
Denote by $\Delta_\s$ the Laplacian operator associated to $q_\s$. The operators
\[
  \epsilon_\s+\frac{1}{2} \dim_\R
  \s,\quad -\frac{1}{2}q_\s,\quad\frac{1}{2}\Delta_\s
\]
form a standard triple, and each of them leaves the space
$\CV_{\s\times E, \CO_i}$ stable. Proposition \ref{negative} says
that $\epsilon_\s+\frac{1}{2} \dim_\R \s$ is semisimple and has
negative eigenvalues on $\CV_{\s\times E, \CO_i}$, and so by
\cite[Lemma 8.A.5.1]{W1}, the map
\[
  \Delta_\s: \CV_{\s\times E,\CO_i}\rightarrow \CV_{\s\times E,\CO_i}
\]
is injective.

Let $f\in \con^{-\xi}_{\chi_E}(\s\times E)$. Applying Lemma
\ref{support} to $f$ and $\CF_E
(f)$, we conclude that under the restriction map
\[
 r_{\s\times E}: \con^{-\xi}_{\chi_E}(\s\times E)\subset \con^{-\xi}(\s\times E;\CN_i\times E)\rightarrow \con^{-\xi}(\s\times E;
\CO_i\times E),
\]
the image \[r_{\s\times E}(f) \in \CV_{\s\times E,\CO_i}.\]

Since $\CF_{\s}(f)\in\con^{-\xi}_{\chi_E}(\s\times E)$ is supported in
\[
 \CN_{i}\times \Gamma_E \subset \textrm{(the null cone of the real quadratic space $\s$)}\times
  E,
\]
we conclude that $f$ and thus $r_{\s\times E}(f)$ are annihilated by some
positive power of $\Delta_\s$. By the injectivity of $\Delta_\s$
on $\CV_{\s\times E,\CO_i}$, we conclude that $r_{\s\times
E}(f)=0$ and we are done.
\end{proof}

The remaining part of this section is devoted to a proof of
Proposition \ref{negative}.

Pick any element $\mathbf e\in\CO_i$ and extend it to a standard triple
$\mathbf{h},\mathbf{e}, \mathbf{f}\in \s$. Then we have a vector
space decomposition
\[
  \s=[\s,\mathbf{e}]\oplus \s^{\mathbf{f}}.
\]

Let $\oU(E)$ act on $\oU(E)\times \s^{\mathbf f} \times E$ via the
left translation on the first factor. Define a $\oU(E)$-equivariant
map
\begin{equation}
    \begin{array}{rcl}
      \theta: \oU(E)\times \s^{\mathbf f} \times E& \rightarrow & \s\times E,\\
          (g, x,v)&\mapsto & g\cdot (x+\mathbf{e}, v).
   \end{array}
\end{equation}

\begin{lemp}
\label{thetar} The vector field
\[
  \iota_{\mathbf{h}/2}+ \epsilon_{\s^{\mathbf{f}},
  1-\ad(\mathbf{h}/2)}-\epsilon_{E,\mathbf{h}/2}
\]
on $\oU(E)\times \s^{\mathbf f} \times E$ is $\theta$-related to
the Euler vector field $\epsilon_{\s}$ on $\s\times E$, where
$\iota_{\mathbf{h}/2}$ is the left invariant vector field on
$\oU(E)$ whose tangent vector at the identity is $\mathbf{h}/2$.
\end{lemp}
\begin{proof}
Since both vector fields under consideration are
$\oU(E)$-invariant, it suffices to prove the $\theta$-relatedness
at a point of the form
\[
 \mathbf{x}:=(1,x,v)\in  \oU(E)\times \s^{\mathbf{f}}\times E.
\]
Under the differential of $\theta$ at $\mathbf{x}$, we have
\[
  \begin{array}{lcl}
      \iota_{\mathbf{h}/2}|_{\mathbf{x}}=(\mathbf{h}/2,0,0)& \mapsto& ([\mathbf{h}/2, x+\mathbf{e}], (\mathbf{h}/2)v),\smallskip\\
      \epsilon_{\s^{\mathbf{f}},1-\ad(\mathbf{h}/2)}|_{\mathbf{x}}=(0,x-[\mathbf{h}/2,x],0) &\mapsto& (x-[\mathbf{h}/2,x],0),\smallskip\\
     \epsilon_{E,\mathbf{h}/2}|_{\mathbf{x}}=(0,0,(\mathbf{h}/2)v)&\mapsto&
     (0,(\mathbf{h}/2)v).
            \end{array}
\]
This implies the lemma since
$\epsilon_{\s}|_{\theta(\mathbf{x})}=(x+\mathbf{e},0)$.
\end{proof}

Let $\oZ(E)$ act on $\s^{\mathbf f} \times E$ and $\oU(E)\times
\s^{\mathbf f} \times E$ via its action on the factor $E$. Then
the map $\theta$ is $\oZ(E)$-equivariant as well. Note that
$\theta$ is submersive from $\oU(E)\times \{0\}\times E$ onto
$\CO_i\times E$ (cf. \cite[Page 299]{W1}). Therefore it yields an
injective pull back map
\[
   \theta^*: \con^{-\xi}( \s\times E;\CO_i\times E)^{\oU(E)}\hookrightarrow
         \con^{-\xi}(\oU(E)\times \s^{\mathbf f} \times E;\oU(E)\times \{0\}\times E)^{\oU(E)\times \oZ(E)}.
\]
Denote by
\[
         r_{\s^{\mathbf f}\times E}: \con^{-\xi}(\oU(E)\times \s^{\mathbf f} \times E;\oU(E)\times \{0\}\times E)^{\oU(E)\times\oZ(E)}\rightarrow
        \con^{-\xi}(\s^{\mathbf f} \times E;\{0\}\times E)^{\oZ(E)}
\]
the linear isomorphism specified by the rule
\begin{equation}
\label{drestriction}
 f=1\otimes  r_{\s^{\mathbf f}\times E}f.
\end{equation}

Recall the space $\CV_{E, \mathbf e}\subset \con^{-\xi}(E)$ from \eqref{dve}.

\begin{lemp}\label{injective}
The composition map $r_{\s^{\mathbf f}\times E}\circ \theta^*$
sends $\CV_{\s\times E,\CO_i}$ into $\con^{-\xi}(\s^{\mathbf f};\{0\})
  \otimes \CV_{E, \mathbf e}$, and the following diagram
\[
   \begin{CD}
           \CV_{\s\times E,\CO_i}\,  &\lhook\joinrel\xrightarrow{r_{\s^{\mathbf f}\times E}\circ\theta^*}
                                         & \,\con^{-\xi}(\s^{\mathbf f};\{0\})
  \otimes \CV_{E, \mathbf e}\\
             @V\epsilon_\s VV            @VV\epsilon_{\s^\mathbf{f},1-\ad(\mathbf{h}/2)}-\epsilon_{E,\mathbf{h}/2}V                \\
           \CV_{\s\times E,\CO_i}\, &\lhook\joinrel\xrightarrow{r_{\s^{\mathbf f}\times E}\circ\theta^*}
                                         & \,\con^{-\xi}(\s^{\mathbf f};\{0\})
  \otimes \CV_{E, \mathbf e}\\
  \end{CD}
\]
commutes.
\end{lemp}
\begin{proof} The first assertion follows by noting that both $\theta^*$ and
$r_{\s^\mathbf{f},E}$ commute with the partial Fourier transform
along $E$. The second assertion follows from Lemma \ref{thetar}.
\end{proof}

\begin{lemp}
\label{eulersmall} The vector field $\epsilon_{\s^\mathbf{f},1-\ad(\mathbf{h}/2)}$ acts
semisimply on $\con^{-\xi}(\s^{\mathbf f};\{0\})$, and all its
eigenvalues are real numbers $<-\frac{1}{2}\dim_\R \s$.
\end{lemp}
\begin{proof}
Recall that $\s$ is assumed to be nonzero. We view $\s$ as a $\sl_2(\R)$-module via the adjoint
representation and the standard triple $\{\mathbf{h},\mathbf{e},
\mathbf{f}\}$. We shall prove that the analog of Lemma
\ref{eulersmall} holds for any finite dimensional nonzero
$\sl_2(\R)$-module $F$. Without loss of generality, we may assume
that $F$ is irreducible of real dimension $k+1$. Then
\[
  \epsilon_{F^\mathbf{f},1-\mathbf{h}/2}=(1+k/2)\epsilon_{F^\mathbf{f}},
\]
which clearly acts semisimply on $\con^{-\xi}(F^{\mathbf
f};\{0\})$, with all its eigenvalues real numbers $\leq -(1+k/2)
=-\frac{1}{2}\dim_\R F-\frac{1}{2}<-\frac{1}{2}\dim_\R F$.
\end{proof}

In view of Lemma \ref{injective}, Proposition \ref{negative} will
follow from Proposition \ref{glk} and Lemma \ref{eulersmall}. We
have thus proved Proposition \ref{indn} when $\CO_i$ is
distinguished.

\section{Reduction within the null cone: metrically proper orbits}
\label{sec:nullcone2}

We are in the same setting as Section \ref{sec:nullcone}, so
$(A,\tau)$ is simple and $E$ is nonzero. Now assume that $\CO_i$ is
not distinguished. The purpose of this section is to prove
Proposition \ref{indn} in this case.

If $F$ is a non-degenerate finite dimensional real quadratic space,
we say that a submanifold $S$ of $F$ is metrically proper if  for
every $x\in S$, the tangent space $\oT_x(S)$ is contained in a
proper non-degenerate subspace of the real quadratic space  $F$.

\begin{lemp}\label{metricp}
The orbit $\CO_i$ is metrically proper in $\s$.
\end{lemp}

\begin{proof}
Let $x\in \CO_i$. By definition, it commutes with a nonzero
semisimple element $h\in \s$. Denote by $\a_h$ the center of $\s^h$
(the centralizer of $h$ in $\s$), which is a nonzero non-degenerate
subspace of $\s$.

Using the fact that every element of $\a_h$ commutes with $x$, we
see that the tangent space
\[
  \oT_x(\CO_i)=[\u(E),x]=[\s, x]
\]
is contained in the proper non-degenerate subspace $(\a_h)^\perp\subset \s$.

\end{proof}

The following lemma is a form of the uncertainty principle.
\begin{lemp}\label{uncert}
Let $M$ be a Nash manifold and let $F$ be a non-degenerate finite dimensional real quadratic space. Let $Z_1\supset Z_2$ be closed  semialgebraic subsets of $F$ so that the difference $Z_1\setminus Z_2$ is a metrically proper submanifold of $F$. Then \[
  \con^{-\xi}(M\times F; M\times Z_1,M\times \Gamma_F)=\con^{-\xi}(M\times F; M\times Z_2,M\times \Gamma_F),
\]
where $\Gamma_F$ is the null cone of $F$.
\end{lemp}
\begin{proof}
This is a direct consequence of \cite[Lemma 2.2]{JSZ2}.
\end{proof}

In view of Lemma \ref{uncert}, Proposition \ref{indn} in the metrically proper case follows by noting that the partial Fourier transform $\CF_{\s}$ preserves the space $\con^{-\xi}_{\chi_E}(\s\times E)$ and that $\CN_i$ is contained in the null cone of the real quadratic space $\s$.

\section{Reduction to the null cone and proof of Theorem \ref{liealg}}
\label{sec:reduction}

Now let $E$ be an $\epsilon$-Hermitian $A$-module, with $(A,\tau)$
arbitrary. Define an involution on $\End_A(E)$, which is still
denoted by $\tau$,  by requiring that
\begin{equation}\label{deftau}
  \la x u,v\ra_E=\la u, x^\tau v\ra_E,\quad x\in \End_A(E),\,u,v\in E.
\end{equation}

For any $x$ in $\End_A(E)$, let $A_x$ be the subalgebra generated by
$x$, $x^\tau$ and scalar multiplications by $A$.  If $x$ is a
semisimple element in $\oU(E)$ or $\u(E)$, then $(A_x,\tau)$ is a
commutative involutive algebra. Write $E_x:=E$, to be viewed as an
$\epsilon$-Hermitian $A_x$-module with the form $\la\,,\,\ra_{E_x}$.
The latter is specified by
\[
    \tr_{A_x/\R}(a\la u, v\ra_{E_x})= \tr_{A/\R}(\la a u, v\ra_{E}), \quad u,v\in E, \,\,a\in A_x.
\]

Write
\[
   (A,\tau)=(A_1,\tau_1)\times (A_2,\tau_2)\times \cdots \times (A_l,\tau_l)
\]
as a product of simple commutative involutive algebras. We then have
\begin{equation}
\label{producte}
  E=E_1\times E_2\times \cdots\times E_l,
\end{equation}
where $E_j:=A_j\otimes_A E$, which is naturally an
$\epsilon$-Hermitian $A_j$-module. Note that $E_j$ is free as an
$A_j$-module. Put
\begin{equation}
\label{dsdimension}
  \sdim(E):=\sum_{j=1}^l \max\{\rank_{A_j}(E_j)-1,\,0\}+\dim_\R (E).
\end{equation}

The following result may be considered as a case of Harish-Chandra descent.

\begin{prp}\label{descent1}
Assume that for all
commutative involutive algebra $A'$ and all $\epsilon$-Hermitian $A'$-module
$E'$,
\begin{equation}\label{vanishep1}
  \sdim(E')<\sdim(E) \quad\textrm{implies}\quad \con^{-\xi}_{\chi_{E'}}(\u(E')\times E')=0.
\end{equation}
Then every $f\in \con^{-\xi}_{\chi_{E}}(\u(E)\times E)$ is
supported in $(\z(E)+\CN_E)\times E$.
\end{prp}

\begin{proof}
Let $x$ be a semisimple element in $\u(E)\setminus \z(E)$. Then $
\sdim(E_x)<\sdim(E)$.

For any $y\in \u(E_x)$, denote by $J(y)$  the determinant of the
$\R$-linear map
\[
   [y, \cdot]: \u(E)/\u(E_x)\rightarrow \u(E)/\u(E_x).
\]
Then $J$ is a $\breve{\oU}(E_x)$-invariant function on $\u(E_x)$.
Put
\[
  \u(E_x)^\circ:=\{y\in \u(E_x)\mid J(y)\neq 0\},
\]
which contains $x+\CN_{E_x}$.  The map
\[
 \begin{array}{rcl}
   \pi _x: \breve{\oU}(E)\times (\u(E_x)^\circ\times E_x)&\rightarrow& \u(E)\times E,\\
           (g,y,v) &\mapsto &g \cdot (y,v)
   \end{array}
\]
is a submersion. Therefore we have a well defined restriction map
(\cite[Lemma 4.4]{JSZ2})
\begin{equation}\label{reep}
  r_{E,E_x}: \con^{-\xi}_{\chi_{E}}(\u(E)\times E)\rightarrow  \con^{-\xi}_{\chi_{E_x}}(\u(E_x)^\circ\times
   E_x),
\end{equation}
which is specified by the rule
\[
   \pi_x^*(f)=\chi _E\otimes  r_{E,E_x}(f).
\]
The assumption (\ref{vanishep1}) easily implies that the latter
space in \eqref{reep} is zero. Thus every $f\in
\con^{-\xi}_{\chi_{E}}(\u(E)\times E)$ vanishes on the image of
$\pi_x$. As $x$ is arbitrary, the proposition follows.
\end{proof}

\begin{prpp}\label{descent2} Assume that $A$ is simple, and for all commutative involutive algebra $A'$ and all $\epsilon$-Hermitian $A'$-module $E'$,
\[
 \sdim(E')<\sdim(E)\quad\textrm{implies}\quad \con^{-\xi}_{\chi_{E'}}(\u(E')\times E')=0.
\]
Then every $f\in \con^{-\xi}_{\chi_{E}}(\u(E)\times E)$ is supported
in $\u(E)\times \Gamma_E$.
\end{prpp}
\begin{proof}
The proof is similar to that of \cite[Proposition 5.2]{AGRS}.
\end{proof}

We are now ready to prove Theorem \ref{liealg}, which will be by induction on $\sdim(E)$. If $\sdim(E)$ is $0$, then $E=0$ and the theorem is trivial. Now assume that $E$ is nonzero and we have proved Theorem \ref{liealg} when $\sdim(E)$ is smaller.

Without loss of generality, assume that $E$ is faithful as an
$A$-module. If $A$ is not simple, then for $1\leq j\leq l$,
\[
  \sdim(E_j)<\sdim(E) \,\,\textrm{ and thus }\,\,\con^{-\xi}_{\chi_{E_j}}(\u(E_j)\times E_j)=0.
\]
This clearly implies that $\con^{-\xi}_{\chi_{E}}(\u(E)\times E)=0$.

Otherwise assume that $A$ is simple.  Note that $\breve{\oU}(E)$ acts trivially on $\z(E)$. Propositions \ref{descent1} and \ref{descent2} imply that every element in $\con^{-\xi}_{\chi_{E}}(\su(E)\times E)$ is supported in $\CN_E\times \Gamma_E$, and Proposition \ref{indn} further implies that $\con^{-\xi}_{\chi_{E}}(\su(E)\times E)=0$. Therefore $\con^{-\xi}_{\chi_{E}}(\u(E)\times E)=0$ and the proof of Theorem \ref{liealg} is now complete.

\section{Proof of Theorem \ref{thm:main}}
\label{sec:proofC}

The argument of this section is standard and we will thus be brief.  As before, let $(A,\tau)$ be a commutative involutive algebra and let $E$ be an $\epsilon$-Hermitian $A$-module.

\begin{thm}\label{liealg2}
One has that $\con^{-\infty}_{\chi_{E}}(\u(E)\times E)=0$.
\end{thm}
\begin{proof}
In view of Theorem \ref{liealg}, this follows from a general principle of ``distributions versus Schwartz distributions" (\cite[Theorem 4.0.2]{AGS1}).
\end{proof}

\begin{thmt}\label{liegp}
One has that $\con^{-\infty}_{\chi_{E}}(\oU(E)\times E)=0$.
\end{thmt}
\begin{proof}
Again we prove by induction on $\sdim(E)$ and assume that the theorem holds when $\sdim(E)$ is smaller. As in the proof of Proposition \ref{descent1}, we show that
\[
  \con^{-\infty}_{\chi_{E}}(\oU(E)\times E)=\con^{-\infty}_{\chi_{E}}(\oU(E)\times E; (\oZ(E)\CU_E)\times E),
\]
where $\CU_E$ is the set of unipotent elements in $\oU(E)$. Note that $\breve{\oU}(E)$ acts on $\oZ(E)$ trivially. The map
\[
    \begin{array}{rcl}
     \rho_E: \oZ(E)\times \su(E)\times E&\rightarrow&  \oU(E)\times E,\\
          (z,x,v)&\mapsto& (z\exp(x),v)
    \end{array}
\]
is $\breve{\oU}(E)$-equivariant and yields an injective pull back map
\[
            \con_{\chi_E}^{-\infty}(\oU(E)\times E; (\oZ(E) \CU_E)\times E)
     \lhook\joinrel\xrightarrow{\rho_E^*}
        \con_{\chi_E}^{-\infty}(\oZ(E)\times \su(E)\times E; \oZ(E)\times
  \CN_E\times E).
\]
The latter space vanishes by Theorem \ref{liealg2} and the result follows.

\end{proof}

Assume for the moment that $(A,\tau)$ is simple and $\epsilon=1$. Let $E=E'\oplus Av_0$ be an orthogonal decomposition with $v_0\notin \Gamma_E$ (the null cone of $E$). Then $\breve{\oU}(E')$ is identified with the stabilizer of $v_0\in E$ in $\breve{\oU}(E)$ via the embedding
\begin{equation}\label{embu}
  (g,\delta)\mapsto \left(\left[\begin{array}{cc}
    \delta g&0\\
    0&\tau_\delta\\
  \end{array}\right], \delta\right),
\end{equation}
where $\tau_\delta: Av_0\rightarrow A v_0$ is the $\R$-linear map
given by
\[
  \tau_\delta(a v_0)=\left\{
                           \begin{array}{ll}
                           a v_0,\quad&\textrm{if }\delta=1,\\
                           -a^\tau v_0,\quad&\textrm{if }\delta=-1.
                           \end{array}
                       \right.
\]

The following result is a consequence of Theorem \ref{liegp} in the case of $\epsilon=1$. We refer the reader to \cite[Proposition 5.1]{AGRS} for the necessary argument.

\begin{cort}\label{liegph}
Let the notation be as in this section. Let $\breve{\oU}(E')$ act on $\oU(E)$ through the action of $\breve{\oU}(E)$. Then  $\con^{-\infty}_{\chi_{E'}}(\oU(E))=0$.
\end{cort}

Corollary \ref{liegph} implies Theorem \ref{thm:main} for the first five classical groups of this paper.

\vsp

Now assume that $(A,\tau)$ is simple and $\epsilon=-1$. Write
\[
   \oH(E):=E\times
   A^{\tau=1}\quad \textrm{(where $A^{\tau=1}$ is the set of $\tau$-fixed elements in $A$)}
\]
for the Heisenberg group with group multiplication
\[
   (u,t)(u',t')=(u+u', t+t'+\frac{\la u,u'\ra_E}{2}-\frac{\la
   u',u\ra_E}{2}).
\]
Let $\breve{\oU}(E)$ act on $\oH(E)$ as group automorphisms by
\[
     (g,\delta)\cdot (u,t):=(gu, \delta t).
\]
We form the semidirect products (the Jacobi groups)
\begin{equation}
\label{defjacobi}
   \breve{\oJ}(E):=\breve{\oU}(E)\ltimes\oH(E)\supset \oJ(E):=\oU(E)\ltimes\oH(E).
\end{equation}

The following result is a consequence of Theorem \ref{liegp} in the case of $\epsilon=-1$. The necessary argument can be found in \cite[Theorem 3.1]{Dijk2} or \cite[Theorem D]{Su09}.

\begin{cort}\label{liegpsh}
Let the notation be as in this section. Let $\breve{\oU}(E)$ act on $\oJ(E)$ by
\[
  g\cdot x:=gx^{\chi_E(g)}g^{-1}.
\]
Then $\con^{-\xi}_{\chi_E}(\oJ(E))=0$.
\end{cort}

Corollary \ref{liegpsh} implies Theorem \ref{thm:main} for all Jacobi groups.

\vsp Finally we come to the special orthogonal groups. The key idea
to establish this variant is due to Waldspurger \cite{Wald09} and it
is to introduce another extended group. Assume that $\epsilon=1$. If
$E$ is a quadratic space (i.e. $A$ is $\R$ or $\C$, and $\tau$ is
trivial), define
\[
   \operatorname{S}\!\breve{\oO}(E):=\left \{(g,\delta)\in \oO(E)\times \{\pm 1\}\mid \det(g)=\delta^{\left[\frac{\dim_A E +1}{2} \right]}\right \}\supset \SO(E).
\]
Denote by $\chi_{\os, E}$ the quadratic character on
$\operatorname{S}\!\breve{\oO}(E)$ with kernel $\SO(E)$. Let
$\operatorname{S}\!\breve{\oO}(E)$ act on $\SO(E)$ and $E$  as in
equations \eqref{actiong} and \eqref{actione}, respectively.

\begin{thmt}\label{liegpso}
One has that $\con^{-\infty}_{\chi_{\os, E}}(\SO(E)\times E)=0$.
\end{thmt}

The descent process in the proof of Theorem \ref{liegpso} requires
us to define a compatible family of extended groups.  First assume
that $(A,\tau)$ is simple. Define
\[
  (\breve \oU_{\rs}(E), \oU_{\rs}(E)):=\begin{cases}
                                            (\operatorname{S}\!\breve{\oO}(E),\SO(E)),\quad &
                                            \textrm{if $\tau$ is
                                            trivial},\smallskip\\
                                            (\breve{\oU}(E),\oU(E)),\quad &
                                            \textrm{if $\tau$ is
                                           non-trivial}.\end{cases}
\]
In general, write $E=E_1\times E_2\times \cdots\times E_l$ as in
\eqref{producte}.  We put
\[
  \oU_{\rs}(E):=\oU_{\rs}(E_1)\times \oU_{\rs}(E_2)\times \cdots \times \oU_{\rs}(E_l),
\]
and
\begin{eqnarray*}
       \breve \oU_{\rs}(E)&:=&\breve \oU_{\rs}(E_1)\times_{\{\pm 1\}} \breve \oU_{\rs}(E_2)\times_{\{\pm 1\}} \cdots \times _{\{\pm 1\}} \breve \oU_{\rs}(E_l)\\
       &:=&\{(g_1,g_2,\cdots,g_l,\delta)\mid (g_j,\delta)\in \breve \oU_{\rs}(E_j),\, j=1,2,\cdots, l \}.
\end{eqnarray*}
The latter contains the former as a subgroup of index two. Let
$\breve \oU_{\rs}(E)$ act on $\oU_{\rs}(E)$ and $E$, again as in
\eqref{actiong} and \eqref{actione}.

In the notation of this paper, Waldspurger's observation may be
stated as follows.

\begin{lemt} Let $x$ be a semisimple element of
$\oU_{\rs}(E)$ and let $E_x$ be as in Section \ref{sec:reduction}.
Then $x\in \oU_{\rs}(E_x)$, and  $\breve \oU_{\rs}(E_x)$ is
contained in the stabilizer of $x$ in $\breve \oU_{\rs}(E)$.
\end{lemt}

\begin{lemt} Assume that $A$ is simple. Let $v\in E\setminus \Gamma_E$. Then the stabilizer of $v$ in $\breve \oU_{\rs}(E)$ is naturally isomorphic
to $\breve \oU_{\rs}(E')$, where $E'$ is the orthogonal complement
of $Av$ in $E$.
\end{lemt}

The argument of this paper, together with the above two lemmas, will
imply Theorem \ref{liegpso}. We skip the details. Theorem
\ref{liegpso} implies the analog of Corollary \ref{liegph} for
special orthogonal groups. Theorem \ref{thm:main} for special
orthogonal groups then follows.

\end{document}